\documentclass[doublespacing]{amsart}
\usepackage{amsmath}
\usepackage{amssymb}
\usepackage{graphicx}
\newtheorem{theorem}{Theorem}

\newtheorem{lemma}{Lemma}

\numberwithin{theorem}{section} \numberwithin{equation}{section}
\numberwithin{lemma}{section}

\begin{document}

\title{Exact convergence order of the $L_r$-quantization error for Markov-type measures}

\author{Sanguo Zhu, Youming Zhou, Yongjian Sheng}
\address{School of Mathematics and Physics, Jiangsu University of Technology,\\
Changzhou 213001, China}\email{sgzhu@jsut.edu.cn}

\begin{abstract}
Let $E$ be a graph-directed set associated with a di-graph $G$. Let $\mu$ be a Markov-type measure on $E$. Assuming a separation condition for $E$, we determine the exact convergence order of the $L_r$-quantization error for $\mu$. This result provides us with accurate information on the asymptotics of the quantization error, especially when the quantization coefficient is infinite.
\end{abstract}
\keywords{Markov-type measures, graph directed fractals, quantization error, convergence order, reducible transition matrix.} \subjclass[2000]{28A80, 28A78} \maketitle

\section{Introduction}

The quantization problem for probability measures consists in nonlinear approximation of a given probability measure with discrete measures in $L_r$-metrics.  We refer to \cite{GL:00} for mathematical foundations of this theory and \cite{BW:82,GN:98,Za:63} for its deep background in information theory and engineering technology. One may see \cite{PK:01} for some more related theoretical results.

In the present paper, we further study the asymptotic quantization errors for Markov-type measures supported on graph-directed fractals. For related results on this topic, see \cite{KZ:15,LJL:01}.

Let $P=(p_{ij})_{N\times N}$ be a row-stochastic matrix, i.e,
$p_{ij}\geq0,1\leq i,j\leq N$, and $\sum_{j=1}^{N}p_{ij}=1,1\leq i\leq N$.
We assume
\begin{equation}
{\rm card}(\{1\leq j\leq N:p_{ij}>0\})\geq2\;\;{\rm for\; all}\;\;1\leq i\leq N.\label{cardpij>0}
\end{equation}
Let $\theta$ denote the empty word and set $\Omega_{0}:=\{\theta\}$. Write
\begin{eqnarray*}
&& \Omega_{1}:=\{1,\ldots N\};\;
\Omega_{k}:=\bigg\{\sigma\in\Omega_{1}^{k}:\prod_{h=1}^{k-1}p_{\sigma_{h}\sigma_{h+1}}>0\bigg\},\; k\geq2;\\
 && \Omega^{*}:=\bigcup_{k\geq0}\Omega_{k},\;\Omega_{\infty}:=\bigg\{\sigma\in\Omega_{1}^{\mathbb{N}}:p_{\sigma_{h}\sigma_{h+1}}>0\;\;{\rm for\; all}\;\; h\geq1\bigg\}.
\end{eqnarray*}
We define $|\sigma|:=k$
for $\sigma\in\Omega_{k}$ and $|\theta|:=0$. For $\sigma=\sigma_{1}\ldots\sigma_{n}\in\Omega_n$
with $n\geq k$ or $\sigma\in\Omega_{\infty}$, we write $\sigma|_{k}:=\sigma_{1}\ldots\sigma_{k}$.
If $\sigma,\omega\in\Omega^{*}$ and $(\sigma_{|\sigma|},\omega_{1})\in\Omega_{2}$,
we define
\[
\sigma\ast\omega=\sigma_{1}\sigma_{2}\ldots\sigma_{|\sigma|}\omega_{1}\ldots\omega_{|\omega|}\in\Omega_{|\sigma|+|\omega|}.
\]
Let $J_{i},1\leq i\leq N$, be non-empty compact subsets of $\mathbb{R}^{t}$
with $J_{i}=\overline{{\rm int}(J_{i})}$ for all $1\leq i\leq N$,
where $\overline{B}$ and ${\rm int}(B)$ respectively denote the
closure and interior in $\mathbb{R}^{t}$ of a set $B\subset\mathbb{R}^{t}$. Let $|A|$ denotes the diameter of a set $A\subset\mathbb{R}^{t}$.
Without loss of generality, we assume that
\[
|J_{i}|=1\;\;{\rm for\; all}\;\;1\leq i\leq N.
\]
Let $(c_{ij})_{N\times N}$ be a non-negative matrix such that $c_{ij}\in[0,1)$, and $c_{ij}>0$
if and only if $p_{ij}>0$  for all $1\leq i,j\leq N$.

We call $J_{i},1\leq i\leq N$,
cylinder sets of order one. For each $1\leq i\leq N$, let $J_{ij},\;(i,j)\in\Omega_{2}$,
be non-overlapping subsets of $J_{i}$ such that $J_{ij}$ is geometrically
similar to $J_{j}$ and
\[
\frac{|J_{ij}|}{|J_{j}|}=c_{ij},\;(i,j)\in\Omega_{2}.
\]
We call these sets cylinder sets of order two.

Assume that
cylinder sets $J_{\sigma},\sigma\in\Omega_{k}$, of order $k$ are defined. Let $J_{\sigma\ast i_{k+1}}$ with $\sigma\ast i_{k+1}\in\Omega_{k+1}$,
be non-overlapping subsets of $J_{\sigma}$ such that $J_{\sigma\ast i_{k+1}}$
is geometrically similar to $J_{i_{k+1}}$ and
\[
\frac{|J_{\sigma\ast i_{k+1}}|}{|J_\sigma|}=c_{\sigma_{|\sigma|}i_{{k+1}}}.
\]

Inductively, cylinder
sets of order $k$ are determined for all $k\geq1$. Then we get a ratio-specified
fractal set $E$ satisfying
\[
E:=\bigcap_{k\geq1}\bigcup_{\sigma\in\Omega_{k}}J_{\sigma}.
\]
This type of sets can be described in terms of directed graphs, so we call $E$ a graph-directed set. Fractal properties of such sets, including Hausdorff dimension and Hausdorff measure, have been well studied by Mauldin and Williams \cite{MW:88}, Edgar and  Mauldin \cite{EM:92}.

Let $(\chi_{i})_{i=1}^{N}$ be an arbitrary probability vector with
$\min_{1\leq i\leq N}\chi_{i}>0$. By Kolmogorov consistency theorem,
there exists a unique probability measure $\widetilde{\mu}$ on $\Omega_{\infty}$
such that for every $k\geq 1$ and $\sigma=\sigma_{1}\ldots\sigma_{k}\in\Omega_{k}$, we have
\[
\widetilde{\mu}([\sigma]):=\chi_{\sigma_{1}}p_{\sigma_{1}\sigma_{2}}\cdots p_{\sigma_{k-1}\sigma_{k}};
\]
where $[\sigma]:=\{\omega\in\Omega_{\infty}:\omega|_{|\sigma|}=\sigma\}$.
Let $\pi$ denote the projection from $\Omega_{\infty}$ to $E$:

\[
\pi(\sigma)=x,\;\;{\rm with}\;\;\{x\}:=\bigcap_{k\geq1}J_{\sigma|_{k}},\;\;{\rm for}\;\;\sigma\in\Omega_{\infty}.
\]
As in \cite{KZ:15}, we assume that there exists some
constant $t\in(0,1)$ such that
\begin{equation}
d(J_{\sigma\ast i_{1}},J_{\sigma\ast i_{2}})\geq t\max\{|J_{\sigma\ast i_{1}}|,|J_{\sigma\ast i_{2}}|\}\label{g4}
\end{equation}
for every $\sigma\in\Omega^{*}$ and distinct
$i_{1},i_{2}\in\Omega_{1}$ with $(\sigma_{|\sigma|},i_{l})\in\Omega_2,l=1,2$.
We call the measure $\mu:=\widetilde{\mu}\circ\pi^{-1}$
a Markov-type measure which satisfies
\begin{eqnarray}
\mu(J_{\sigma})=\chi_{\sigma_{1}}p_{\sigma_{1}\sigma_{2}}\cdots p_{\sigma_{k-1}\sigma_{k}}\;\;{\rm for}\;\;\sigma=\sigma_{1}\ldots\sigma_{k}\in\Omega_{k}.\label{markovmeasure}
\end{eqnarray}

Next, let us recall some objects in quantization theory. We set
\[
\mathcal{D}_{n}:=\{\alpha\subset\mathbb{R}^{t}:1\leq{\rm card}(\alpha)\leq n\},\;\;n\in\mathbb{N}.
\]
Let $\nu$ be a Borel probability measure on
$\mathbb{R}^{t}$. For each $n\geq1$, the $n$th quantization error
for $\nu$ of order $r$ is defined by
\begin{eqnarray}
e_{n,r}(\nu):=\bigg(\inf_{\alpha\in\mathcal{D}_{n}}\int d(x,\alpha)^{r}d\nu(x)\bigg)^{\frac{1}{r}},\label{quanerrordef}
\end{eqnarray}
where $d(x,\alpha):=\inf_{a\in\alpha}d(x,a)$ and $d$
is the metric induced by a norm on $\mathbb{R}^{t}$.
For $r\geq1$, $e_{n,r}(\nu)$ agrees with the error in the approximation
of $\nu$ by discrete probability measures supported on at most $n$
points, in $L_{r}$-metrics \cite{GL:00}.

The upper and lower quantization dimension for $\nu$ of order $r$ as defined below are natural characterizations of the convergence rate of $e_{n,r}(\nu)$:
\begin{eqnarray*}
\overline{D}_{r}(\nu):=\limsup_{n\to\infty}\frac{\log n}{-\log e_{n,r}(\nu)},\;\underline{D}_{r}(\nu):=\liminf_{n\to\infty}\frac{\log n}{-\log e_{n,r}(\nu)}.
\end{eqnarray*}
If $\overline{D}_{r}(\nu)=\underline{D}_{r}(\nu)$, we denote the common value by $D_{r}(\nu)$ and call it the quantization dimension for $\nu$ of order $r$.

For $s>0$, we define the $s$-dimensional
upper and lower quantization coefficient for $\nu$ of order $r$
by
\[
\overline{Q}_{r}^{s}(\nu):=\limsup_{n\to\infty}n^{\frac{r}{s}}e^{r}_{n,r}(\nu),\;\;\underline{Q}_{r}^{s}(\nu):=\liminf_{n\to\infty}n^{\frac{r}{s}}e^{r}_{n,r}(\nu).
\]
The upper (lower) quantization dimension is the critical point
at which the upper (lower) quantization coefficient jumps from zero
to infinity \cite{GL:00,PK:01}. When $\underline{Q}_{r}^{s}(\nu)$ and $\overline{Q}_{r}^{s}(\nu)$ are both positive and finite, one can easily see that
$e^{r}_{n,r}(\nu)$ is of the same order as $n^{-\frac{r}{s}}$.

In the remaining part of this section, we recall some concepts regarding digraphs and some previous work in \cite{KZ:15}; then we state our main result of the present paper.

Let $G$ be a directed graph with
vertices $1,2,\ldots,N$; we assume that there exists exactly one edge from $i$ to $j$ if and
only if $p_{ij}>0$; otherwise there is no edge from $i$ to $j$. As in \cite{KZ:15}, we denote by $G=\{1,\ldots,N\}$
both the directed graph and its vertex set. We write
\[
b_{ij}(s):=(p_{ij}c_{ij}^{r})^{\frac{s}{s+r}},\;\; A_G(s):=(b_{ij}(s))_{N\times N}.
\]
Let $\Psi_{G}(s)$ denote the spectral radius of $A_G(s)$. As we noted in \cite{KZ:15},
there exists a unique positive number $s_{r}$ such that $\Psi_{G}(s_r)=1$.

An element $i_{1}\ldots i_{k}\in\Omega_{k}$
is called a path in $G$. We call $H\subset G$, with edges inherited from
$G$, a subgraph of $G$. A subgraph $H$ of $G$ is called strongly
connected if for very pair $i_{1},i_{2}\in H$, there exists a path
$\gamma$ in $H$ which starts at $i_{1}$ and terminates at $i_{2}$. A
strongly connected component of $G$ means a maximal strongly
connected subgraph. We denote by ${\rm SC}(G)$ the set of all strongly
connected components of $G$.

For $H_{1},H_{2}\in{\rm SC}(G)$, we
write $H_{1}\prec H_{2}$, if there is a path $\gamma=i_{1}\ldots i_{k}$
in $G$ such that $i_{1}\in H_{1}$ and $i_{k}\in H_{2}$. If
neither $H_{1}\prec H_{2}$ nor $H_{2}\prec H_{1}$, then we say that $H_{1},H_{2}$
are incomparable.

For $H\in{\rm SC}(G)$, we denote by $A_H(s)$ the sub-matrix
$(b_{ij}(s))_{i,j\in H}$ of $A_{G}(s)$. Let $\Psi_{H}(s)$ be the
spectral radius of $A_H(s)$ and $s_{r}(H)$ be the unique positive
number satisfying $\Psi_{H}(s_{r}(H))=1$. By \cite{KZ:15},
we have
\[
s_{r}=\max_{H\in SC(G)}s_{r}(H).
\]
For every $r\in (0,\infty)$, we write
\[
\mathcal{M}_r:=\{H\in {\rm SC}(G):s_{r}(H)=s_{r}\},\; M_r:={\rm card}(\mathcal{M}_r).
\]

Assume that (\ref{cardpij>0}) and (\ref{g4}) are
satisfied. Let $\mu$ be as defined in (\ref{markovmeasure}).
It is proved in \cite{KZ:15} that

(a) $D_{r}(\mu)=s_{r}$ and $\underline{Q}_{r}^{s_{r}}(\mu)>0$;

(b) $\overline{Q}_{r}^{s_{r}}(\mu)<\infty$ if and only if elements of
$\mathcal{M}_r$ are pairwise
incomparable; otherwise, we have $\underline{Q}_{r}^{s_{r}}(\mu)=\infty$.

When $G$ is strongly connected, Lindsay determined the quantization dimension in terms of the temperature function of the corresponding dynamical systems and proved that the upper and lower quantization coefficient are positive and finite \cite{LJL:01}.

For two sequences $(a_{n})_{n=1}^{\infty}$ and $(b_{n})_{n=1}^{\infty}$ of positive
numbers, we write $a_{n}\lesssim b_{n}$ if there is some constant
$B$ independent of $n$ such that $a_{n}\leq B\cdot b_{n}$. If $a_{n}\lesssim b_{n}$
and $b_{n}\lesssim a_{n}$ we write $a_{n}\asymp b_{n}$. Then, if $\mathcal{M}_r$ consists of
incomparable elements, by (a) and (b), the convergence order of $e_{n,r}^r(\mu)$ is known:
\begin{equation}\label{formu1}
e_{n,r}^r(\mu)\asymp n^{-\frac{r}{s_r}}.
\end{equation}
 When $\mathcal{M}_r$ contains comparable elements, by (b), we have that $\underline{Q}_{r}^{s_{r}}(\mu)=\infty$. However, this does not provide us with accurate information on the asymptotics of the quantization error for $\mu$. As our main result of the paper, we will determine the exact asymptotic order of $e_{n,r}^r(\mu)$ in case that $\underline{Q}_{r}^{s_{r}}(\mu)=\infty$. For a path $\gamma=\gamma_1\cdots\gamma_{|\gamma|}\in\Omega_{|\gamma|}$, we define
\begin{equation*}
T_r(\gamma):={\rm card}(\{H\in\mathcal{M}_r:\gamma_i\in H\;\;{\rm for\;some}\;\;i\}).
\end{equation*}
Clearly, $0\leq T_r(\gamma)\leq M_r$ for all $\gamma\in\Omega^*$. Set $T_r:=\max_{\gamma\in\Omega^*}T_r(\gamma)$. Then we have $1\leq T_r\leq M_r$.  We will prove
\begin{theorem}\label{mthm1}
Assume that (\ref{cardpij>0}) and (\ref{g4}) are
satisfied and let $\mu$ be the Markov-type measure as defined in (\ref{markovmeasure}). We have
\begin{equation}\label{formu2}
e_{n,r}^r(\mu)\asymp n^{-\frac{r}{s_r}}\cdot(\log n)^{(T_r-1)(1+\frac{r}{s_r})}.
\end{equation}
\end{theorem}

If $T_r=1$, then $\mathcal{M}_r$ consists of incomparable elements and (\ref{formu2}) degenerates to (\ref{formu1}). Hence, we assume that $T_r\geq 2$ in the remaining part of the paper.
\section{Preliminaries}

For every $k\geq2$ and $\sigma=\sigma_{1}\ldots\sigma_{k}\in\Omega_{k}$,
we write
\begin{eqnarray*}
\sigma^{-}:=\sigma|_{k-1};\;\; p_{\sigma}:=\prod_{h=1}^{k-1}p_{\sigma_{h}\sigma_{h+1}},\; c_{\sigma}:=\prod_{h=1}^{k-1}c_{\sigma_{h}\sigma_{h+1}}.
\end{eqnarray*}
If $|\sigma|=1$, we set $\sigma^{-}=\theta$; we also define $p_{\sigma}:=1,c_{\sigma}:=1$ for
$\sigma\in\Omega_{1}\cup\left\{ \theta\right\} $. If $\sigma,\omega\in\Omega^{*}$
satisfy $|\sigma|\leq|\omega|$ and $\sigma=\omega|_{|\sigma|}$,
then we write $\sigma\prec\omega$.
We say that two words $\sigma,\omega\in\Omega^{*}$ are incomparable
if neither $\sigma\prec\omega$, nor $\omega\prec\sigma$. We call a finite
subset $\Gamma$ of $\Omega^{*}$ a finite antichain if $\Gamma$ consists of pairwise
incomparable words; a finite antichain $\Gamma$
is said to be maximal, if for every word $\tau\in\Omega_{\infty}$, there exists some word $\sigma\in\Gamma$ such that $\sigma\prec\tau$. Set
\begin{eqnarray*}
\underline{p}:=\min_{(i,j)\in\Omega_{2}}p_{ij},\;\underline{c}:=\min_{(i,j)\in\Omega_{2}}c_{ij},\;\overline{p}:=\max_{(i,j)\in\Omega_{2}}p_{ij},\;\overline{c}:=\max_{(i,j)\in\Omega_{2}}c_{ij}.
\end{eqnarray*}
For $r>0$, let $\underline{\eta}_r:=\underline{p}\underline{c}^{r}$. For every $k\in\mathbb{N}$, we define
\begin{eqnarray}\label{s3}
\Lambda_{k,r}:=\{\sigma\in\Omega^{*}:p_{\sigma^{-}}c_{\sigma^{-}}^{r}\geq\underline{\eta}_r^{k}>p_{\sigma}c_{\sigma}^{r}\};\;
\phi_{k,r}:={\rm card}(\Lambda_{k,r}).
\end{eqnarray}
Then $(\Lambda_{k,r})_{k=1}^{\infty}$ is a sequence of finite maximal
antichains. Write
\begin{eqnarray*}
l_{1k}:=\min_{\sigma\in\Lambda_{k,r}}|\sigma|,\; l_{2k}:=\max_{\sigma\in\Lambda_{k,r}}|\sigma|,\;k\geq 1.
\end{eqnarray*}
With Lemma 2.2 in  \cite{KZ:15}, we have showed that
\begin{eqnarray}
e_{\phi_{k,r},r}^{r}(\mu)\asymp\sum_{\sigma\in\Lambda_{k,r}}p_{\sigma}c_{\sigma}^{r}.\label{characterization}
\end{eqnarray}

We will also need the following estimates of the order of $l_{1k},l_{2k}$ and $\log\phi_{k,r}$:
\begin{lemma}\label{zz9}
For every $r>0$, we have
\begin{eqnarray}
l_{1k},l_{2k}\asymp k\;\;{\rm and}\;\;\log\phi_{k,r}\asymp k.
\end{eqnarray}
\end{lemma}
\begin{proof}
Set $\overline{\eta}_r:=\overline{p}\overline{c}^{r}$. By (\ref{s3}), one can easily see that
\begin{eqnarray}
\underline{\eta}_r^{l_{1k}-1}\leq\underline{\eta}_r^{k}\leq\overline{\eta}_r^{l_{2k}-2}.
\end{eqnarray}
This implies that $l_{1k},l_{2k}\asymp k$. To see the remaining part of the lemma, for every $k\geq 1$,
let $t_{k,r}$ be the unique positive number satisfying
\[
\sum_{\sigma\in\Lambda_{k,r}}(p_{\sigma}c_{\sigma}^{r})^{\frac{t_{k,r}}{t_{k,r}+r}}=1.
\]
 By the definitions in (\ref{s3}), we have
\begin{eqnarray}\label{s17}
&\phi_{k,r}\underline{\eta}_r^{\frac{(k+1)t_{k,r}}{t_{k,r}+r}}\leq1\leq\phi_{k,r}\underline{\eta}_r^{\frac{kt_{k,r}}{t_{k,r}+r}},\;\;k\geq1;
\\&\underline{\eta}_r^{\frac{kr}{t_{k,r}+r}}\underline{\eta}_r\leq\underline{\eta}_r^{\frac{(k+1)r}{t_{k,r}+r}}\leq(p_{\sigma}c_{\sigma}^{r})^{\frac{r}{t_{k,r}+r}}\leq\underline{\eta}_r^{\frac{kr}{t_{k,r}+r}},\;\;\sigma\in\Lambda_{k,r}.\label{s18}
\end{eqnarray}
Using (\ref{characterization}), (\ref{s17}) and (\ref{s18}), we deduce
\begin{eqnarray}\label{g2}
e_{\phi_{k,r},r}^{r}(\mu)&\asymp&\sum_{\sigma\in\Lambda_{k,r}}p_{\sigma}c_{\sigma}^{r}=\sum_{\sigma\in\Lambda_{k,r}}(p_{\sigma}c_{\sigma}^{r})^{\frac{t_{k,r}}{t_{k,r}+r}}(p_{\sigma}c_{\sigma}^{r})^{\frac{r}{t_{k,r}+r}}\nonumber\\
&\asymp&\underline{\eta}_r^{\frac{kr}{t_{k,r}+r}}\asymp\phi_{k,r}^{-\frac{r}{t_{k,r}}}.
\end{eqnarray}
Hence, $t_{k,r}$ converges to $D_r(\mu)=s_r$ as $k\to\infty$. Thus for large $k$, we have
\[
\frac{s_r}{2}\leq t_{k,r}\leq 2s_r\;\;{\rm and}\;\;\xi_r:=\frac{s_r}{s_r+2r}\leq\frac{t_{k,r}}{t_{k,r}+r}\leq \frac{2s_r}{2s_r+r}=:\zeta_r.
\]
By this and (\ref{s17}), we deduce that $1\leq\phi_{k,r}\underline{\eta}_r^{k\xi_r}$ and $\phi_{k,r}\underline{\eta}_r^{k\zeta_r}\leq \underline{\eta}_r^{-\zeta_r}$.
\[
k(\xi_r\log\underline{\eta}_r^{-1})\leq\log\phi_{k,r}\leq k\zeta_r\log\underline{\eta}_r^{-1}+\zeta_r\log\underline{\eta}_r^{-1}\leq k(2\zeta_r\log\underline{\eta}_r^{-1})
\]
for large $k$. This completes the proof of the lemma.
\end{proof}

Next we recall some notations and basic facts related to strongly connected components of $G$.
For every $H\in{\rm SC}(G)$, we write $H^{*}:=\bigcup_{k=1}^{\infty}H^{k}$ and
\begin{eqnarray*}
H_{k}(i) & := & \left\{ \sigma\in H^{k}:\sigma_{1}=i\right\} ,\; H^{*}(i):=\bigcup_{k=1}^{\infty}H_{k}(i);
\;i\in H.
\end{eqnarray*}
Maximal antichains in $H^{*}$ or $H^{*}(i)$ are defined in the same manner as we did for those in $\Omega^{*}$. By Lemma 3.5 in \cite{KZ:15}, there exist constants $M_0,M_1$ such that
\begin{eqnarray}\label{s7}
M_0\leq\sum_{\sigma\in\Gamma}(p_\sigma c_\sigma^r)^{\frac{s_r}{s_r+r}}\leq M_1.
\end{eqnarray}
for every $H\in \mathcal{M}_r$ and every finite maximal antichain $\Gamma$ in $H^*$ or $H^*(i)$.

For $k\geq1$ and a vector $w=(w_{i})_{i=1}^{k}\in\mathbb{R}^{k}$,
we define
\begin{eqnarray}\label{s6}
\overline{w}:=\max_{1\leq i\leq k}w_{i},\;\underline{w}:=\min_{1\leq i\leq k}w_{i}.
\end{eqnarray}
For $H\in\mathcal{M}_r$, denote by $c_{ij}^{(h)}(H)$ the $(i,j)$-entry of $A^h_{H,s_r}$. We have
\begin{lemma}\label{zz7}
There exist constants $C_1,C_2$ such that for  $h\geq 1$, we have
\[
C_1\leq\sum_{j\in H}c_{jp}^{(h)}(H)\leq C_2,\;\;{\rm for}\;\; H\in\mathcal{M}_r\;\;{\rm and} \;\;p\in H.
\]
\end{lemma}
\begin{proof}
Assume that $H\in\mathcal{M}_r$ and ${\rm card}(H)=m$. $H$ is strongly connected, so $A_{H,s_r}$ is irreducible. Let $\xi_H=(\xi_{H,i})_{i=1}^m$ be the unique normalized positive left eigenvector of $A_{H,s_r}$ with respect to Perron-Frobenius eigenvalue $1$. Then
\[
\xi_HA^h_{H,s_r}=\xi_H,\;\;{\rm implying}\;\;\sum_{j\in H}\xi_{H,j}c_{jp}^{(h)}(H)=\xi_{H,p}.
\]
Hence, using the notations in (\ref{s6}), we have
\[
\underline{\xi_H}/\overline{\xi_H}\leq\sum_{j\in H}c_{jp}^{(h)}(H)\leq\overline{\xi_H}/\underline{\xi_H}.
\]
It suffices to set $C_1:=\min_{H\in\mathcal{M}_r}\underline{\xi_H}/\overline{\xi_H}$ and $C_2:=\max_{H\in\mathcal{M}_r}\overline{\xi_H}/\underline{\xi_H}$.
\end{proof}
Let $F:=G\setminus\bigcup_{H\in\mathcal{M}_r}H$. It may happen that $F=\emptyset$. If $F\neq\emptyset$, we set
\begin{eqnarray*}
F_{k}:=\{\sigma\in\Omega_{k}:\sigma_{h}\in F,1\leq h\leq k\},\; k\geq1;\;\; F^{*}:=\bigcup_{k=0}^{\infty}F_{k}.
\end{eqnarray*}
\begin{lemma}\label{zz6}(\cite[Lemma 3.8]{KZ:15})
There exists a constant $t\in(0,1)$ such that
\[
\sum_{\sigma\in F_{n}}\left(p_{\sigma}c_{\sigma}^{r}\right)^{s_{r}/\left(s_{r}+r\right)}\lesssim t^{n}\;\;{\rm for\;large}\;\;n\in\mathbb{N}.
\]
As a consequence, we have
$\sum_{\sigma\in F^*}\left(p_{\sigma}c_{\sigma}^{r}\right)^{s_{r}/\left(s_{r}+r\right)}\lesssim 1$.
\end{lemma}

\section{Proof of Theorem  \ref{mthm1}}
For $\gamma\in\Omega^*$, we have either $T_r(\gamma)=0$, which implies that $\gamma$ does not pass any $H\in\mathcal{M}_r$, or $T_r(\gamma)=l$ for some $1\leq l\leq T_r$. In the latter case, there exist some $H_i\in\mathcal{M}_r,1\leq i\leq l$, such that $H_1\prec H_2\prec\cdots\prec H_l$. We write
\[
\mathcal{H}_l:=\{(H_1,H_2,\cdots, H_l):H_1\prec H_2\prec\cdots\prec H_l,H_i\in\mathcal{M}_r,1\leq i\leq l\}.
\]
By the strong connectedness of $H_i\in\mathcal{M}_r$, we can see that
\begin{eqnarray}\label{s16}
0\leq
{\rm card}(\mathcal{H}_l)\leq \binom{T_r}{l}\;\;{\rm for\; all}\;\; 1\leq l\leq T_r,
\end{eqnarray}
where $\binom{T_r}{l}$ denotes the combination number of choosing $l$ objects out of $T_r$.

If a path $\gamma$ passes $H_1,\cdots, H_l\in\mathcal{M}_r$ and $T_r(\gamma)=l$, $\gamma$ takes the following form:
\begin{eqnarray}\label{s1}
\gamma=\tau_\gamma^{(0)}\ast\sigma_\gamma^{(1)}\ast\tau_\gamma^{(1)}\ast\sigma_\gamma^{(2)}\ast\cdots\ast\tau_\gamma^{(l-1)}\ast\sigma_\gamma^{(l)}\ast\tau_\gamma^{(l)},
\end{eqnarray}
where $\tau^{(i)}\in\mathcal{F}^*,0\leq i\leq l$, and $\sigma^{(i)}\in H_i\in\mathcal{M}_r,1\leq i\leq l$. Let us denote by $\Gamma(H_1,\cdots,H_l)$ the set of all such words $\gamma$, which have entries in
each of $H_1,\cdots,H_l$, but do not have entries in any other elements of $\mathcal{M}_r$. We write
\begin{eqnarray*}
\Lambda_{k,r}(H_1,\cdots,H_l):=\Lambda_{k,r}\cap\Gamma(H_1,\cdots,H_l).
\end{eqnarray*}
Then for large $k$, $\Lambda_{k,r}(H_1,\cdots,H_l)$ is non-empty. We write
\begin{equation}\label{s2}
\lambda_{k,r}((H_i)_{i=1}^l):=\sum_{\gamma\in\Lambda_{k,r}(H_1,\cdots,H_l)}(p_\gamma c_\gamma^r)^{\frac{s_r}{s_r+r}}.
\end{equation}

For the proof of Theorem \ref{mthm1}, we need to estimate the asymptotic order of $\lambda_{k,r}((H_i)_{i=1}^l)$. We divide the estimation into several lemmas. First we give an upper estimate for $\lambda_{k,r}((H_i)_{i=1}^l)$.

\begin{lemma}\label{zz4}
Let $\lambda_{k,r}((H_i)_{i=1}^l)$ be as defined in (\ref{s2}). For $2\leq l\leq T_r$, we have
\begin{equation*}
\lambda_{k,r}((H_i)_{i=1}^l)\lesssim k^{l-1}.
\end{equation*}
\end{lemma}
\begin{proof}
Fix a $\gamma\in\Lambda_{k,r}(H_1,\cdots,H_l)$ of the form (\ref{s1}). We may assume that
\[
{\rm card}(H_i)=m_i,\;1\leq i\leq l.
\]
Let $g_{pq}(i)$ be the $(p,q)$-entry of the matrix $A_{H_i,s_r}^{|\sigma_\gamma^{(i)}|-1}$ when $|\sigma_\gamma^{(i)}|\geq 2$.
We have
\begin{eqnarray}
\sum_{i=1}^l|\sigma_\gamma^{(i)}|\leq l_{2k}-\sum_{i=0}^l|\gamma^{(i)}|\leq l_{2k}.
\end{eqnarray}
For each $1\leq i\leq l$, we denote by $c_i,d_i$ the first and last entry of the word $\sigma_\gamma^{(i)}$. By Lemma \ref{zz7}, for every $h\geq 2$, we have
\begin{eqnarray}\label{s15}
\sum_{|\sigma_\gamma^{(i)}|=h}(p_{\sigma_\gamma^{(i)}}c_{\sigma_\gamma^{(i)}}^r)^{\frac{s_r}{s_r+r}}=g_{c_id_i}(i)\leq \sum_{j=1}^{m_i}g_{jd_i}(i)\leq C_2.
\end{eqnarray}
If $|\sigma_\gamma^{(i)}|=1$, we have $(p_{\sigma_\gamma^{(i)}}c_{\sigma_\gamma^{(i)}}^r)^{\frac{s_r}{s_r+r}}=1$.

Now we fix $\tau^{(i)}\in\mathcal{F}^*,0\leq i\leq l$; and $c_i,d_i\in H_i,1\leq i\leq l$. We denote by
$\Lambda^{(1)}_{k,r}(H_1,\cdots,H_l)$ the set of words $\gamma$ in $\Lambda_{k,r}(H_1,\cdots,H_l)$ such that
\[
\tau_\gamma^{(i)}=\tau^{(i)},\;0\leq i\leq l;\;\;(\sigma_\gamma^{(i)})_1=c_i,\;\;(\sigma_\gamma^{(i)})_{|\sigma_\gamma^{(i)}|}=d_i,\;\;1\leq i\leq l.
\]
Let $I^{(1)}_{k,r}$ be the set of all $(\sigma^{(1)},\ldots,\sigma^{(l-1)})$ such that for some
\[
\gamma\in\Lambda^{(1)}_{k,r}(H_1,\cdots,H_l),
\]
$\sigma_\gamma^{(i)}=\sigma^{(i)},1\leq i\leq l-1$.
We further fix $(\sigma^{(1)},\ldots,\sigma^{(l-1)})\in I^{(1)}_{k,r}$, and write
\[
\Lambda^{(2)}_{k,r}(H_1,\cdots,H_l):=\{\gamma\in\Lambda^{(1)}_{k,r}(H_1,\cdots,H_l):\sigma_\gamma^{(i)}=\sigma^{(i)},1\leq i\leq l-1\}.
\]
We denote by $D_{k,r}$ the set of the corresponding $\sigma_\gamma^{(l)}$, namely,
\begin{eqnarray}\label{s9}
D_{k,r}:=\{\sigma\in H_l^*:\gamma\in\Lambda^{(2)}_{k,r}(H_1,\cdots,H_l),\sigma_\gamma^{(l)}=\sigma\}.
\end{eqnarray}
Then by the proof of Proposition 3.9 of \cite{KZ:15}, $D_{k,r}$ is contained in the union of $M_2$ finite maximal antichains in $H_l^*(d_l)$, where
\[
M_2:=\inf\left\{ h\in\mathbb{N}:(\overline{p}\,\overline{c}^{r})^h<\underline{\eta}\right\}+1.
\]
Thus, by (\ref{s7}), for $M_3:=M_1M_2$, we have
\begin{eqnarray}
&&\sum_{\gamma\in\Lambda^{(2)}_{k,r}(H_1,\cdots,H_l)}(p_\gamma c_\gamma^r)^{\frac{s_r}{s_r+r}}\nonumber\\&&\leq\prod_{i=0}^l
(p_{\tau^{(i)}}c_{\tau^{(i)}}^r)^{\frac{s_r}{s_r+r}}\prod_{i=1}^{l-1}
(p_{\sigma^{(i)}}c_{\sigma^{(i)}}^r)^{\frac{s_r}{s_r+r}}\sum_{\gamma\in D_{k,r}}(p_{\sigma_\gamma^{(l)}}c_{\sigma_\gamma^{(l)}}^r)^{\frac{s_r}{s_r+r}}\nonumber\\&&\leq M_3\prod_{i=0}^l
(p_{\tau^{(i)}}c_{\tau^{(i)}}^r)^{\frac{s_r}{s_r+r}}\prod_{i=1}^{l-1}
(p_{\sigma^{(i)}}c_{\sigma^{(i)}}^r)^{\frac{s_r}{s_r+r}}.\label{zz1}
\end{eqnarray}
Note that $1\leq|\sigma_\gamma^{(i)}|\leq l_{2k}$ for all $1\leq i\leq l$. Hence, by (\ref{s15}), we have
\begin{eqnarray*}
&&\sum_{(\sigma^{(1)},\ldots,\sigma^{(l-1)})\in I^{(1)}_{k,r}}
\prod_{i=1}^{l-1}(p_{\sigma^{(i)}}c_{\sigma^{(i)}}^r)^{\frac{s_r}{s_r+r}}\\&&\leq\prod_{i=1}^{l-1}\sum_{h=1}^{l_{2k}}\sum_{|\sigma^{(i)}|=h}(p_{\sigma^{(i)}}c_{\sigma^{(i)}}^r)^{\frac{s_r}{s_r+r}}\leq\prod_{i=1}^{l-1}\sum_{h=1}^{l_{2k}}
g^{(h)}_{c_id_i}\leq (l_{2k}\widetilde{C}_2)^{l-1},
\end{eqnarray*}
where $\widetilde{C}_2:=\max\{C_2,1\}$. Using this and (\ref{zz1}), we deduce
\begin{eqnarray}
&&\sum_{\gamma\in\Lambda^{(1)}_{k,r}(H_1,\cdots,H_l)}(p_\gamma c_\gamma^r)^{\frac{s_r}{s_r+r}}\nonumber
\\&&=\sum_{(\sigma^{(1)},\ldots,\sigma^{(l-1)})\in I^{(1)}_{k,r}}\sum_{\gamma\in\Lambda^{(2)}_{k,r}(H_1,\cdots,H_l)}(p_\gamma c_\gamma^r)^{\frac{s_r}{s_r+r}}\nonumber\\&&\leq M_3\prod_{i=0}^l
(p_{\tau^{(i)}}c_{\tau^{(i)}}^r)^{\frac{s_r}{s_r+r}}\sum_{(\sigma^{(1)},\ldots,\sigma^{(l-1)})\in I^{(1)}_{k,r}}
\prod_{i=1}^{l-1}(p_{\sigma^{(i)}}c_{\sigma^{(i)}}^r)^{\frac{s_r}{s_r+r}}
\nonumber\\&&\leq M_3 \widetilde{C}_2^{l-1}l_{2k}^{l-1}\prod_{i=0}^l
(p_{\tau^{(i)}}c_{\tau^{(i)}}^r)^{\frac{s_r}{s_r+r}}.\label{s8}
\end{eqnarray}
Let $I^{(0)}_{k,r}$ denote the set of all possible $(\tau^{(0)},\ldots,\tau^{(l)})$ such that $H_i,1\leq i\leq l$, are connected via $\tau^{(i)},0\leq i\leq l$ (cf. (\ref{s1})). Note that the number of possible choices of $(c_i,d_i),1\leq i\leq l$, is bounded from above by $N^{2l}$. Using this and (\ref{s8}), we deduce
\begin{eqnarray*}
\lambda_{k,r}((H_i)_{i=1}^l)&=&\sum_{\gamma\in\Lambda_{k,r}(H_1,\cdots,H_l)}(p_\gamma c_\gamma^r)^{\frac{s_r}{s_r+r}}\\&\leq& N^{2l}\sum_{(\tau^{(0)},\ldots,\tau^{(l)})\in I^{(0)}_{k,r}}\sum_{\gamma\in\Lambda^{(1)}_{k,r}(H_1,\cdots,H_l)}(p_\gamma c_\gamma^r)^{\frac{s_r}{s_r+r}}
\\&\leq& N^{2l}M_3\widetilde{C}_2^{l-1}\cdot l_{2k}^{l-1}\sum_{(\tau^{(0)},\ldots,\tau^{(l)})\in I^{(0)}_{k,r}}\prod_{i=0}^l
(p_{\tau^{(i)}}c_{\tau^{(i)}}^r)^{\frac{s_r}{s_r+r}}\\&\leq& N^{2l}M_3\widetilde{C}_2^{l-1}\cdot l_{2k}^{l-1}\bigg(\sum_{\gamma\in F^*}(p_\tau c_\tau^r)^{\frac{s_r}{s_r+r}}\bigg)^{l+1}.
\end{eqnarray*}
This, together with Lemmas \ref{zz6}, \ref{zz9}, implies
\begin{eqnarray*}
\lambda_{k,r}((H_i)_{i=1}^l)\lesssim l_{2k}^{l-1}\asymp k^{l-1}.
\end{eqnarray*}
\end{proof}

Next, we give a lower estimate for $\lambda_{k,r}((H_i)_{i=1}^l)$. For $2\leq l\leq T_r$, let
\begin{eqnarray}
q_1:=l_{1k}-3lN-1.\label{s5}
\end{eqnarray}
For $0\leq p_1\leq q_1$, we define $q_2:=q_1-p_1$. Then for $0\leq p_2\leq q_2$, we set $q_3:=q_2-p_2$. When $p_h,q_h$ are determined for all $1\leq h\leq i-1\leq l-1$, we set
\begin{eqnarray}
q_i:=q_{i-1}-p_{i-1}=q_1-\sum_{h=1}^{i-1}p_h.\label{s6}
\end{eqnarray}
Then $p_i$ is allowed to take values in $[0,q_i]\cap\mathbb{N}$. Note that, for every $2\leq i\leq l-1$, $q_i$ is dependent upon the choices of $p_h,1\leq h\leq i-1$.
\begin{lemma}\label{g1}
For $l\geq 2$ and large $k$, we have
\begin{equation}\label{zz2}
\sum_{p_1=0}^{q_1}\sum_{p_2=0}^{q_2}\cdots\sum_{p_{l-1}=0}^{q_{l-1}}1\gtrsim k^{l-1}.
\end{equation}
\end{lemma}
\begin{proof}
The sum on the left of (\ref{zz2}) equals a combination number. We can think of this as follows. We distribute $q_1$ objects among $l$ (not $l-1$) people, allowing that some people have no objects. The number all possible ways of such divisions is not less than $\binom{q_1}{l-1}$. Thus, by Lemma \ref{zz9}, we have
\begin{equation}\label{zz3}
\sum_{p_1=0}^{q_1}\sum_{p_2=0}^{q_2}\cdots\sum_{p_{l-1}=0}^{q_{l-1}}1\geq \binom{q_1}{l-1}\gtrsim l_{1k}^{l-1}\asymp k^{l-1}.
\end{equation}
This completes the proof of the lemma.
\end{proof}

\begin{lemma}\label{zz8}
Let $\lambda_{k,r}((H_i)_{i=1}^l)$ be as defined in (\ref{s2}). For $2\leq l\leq T_r$, we have
\begin{equation}\label{s4}
\lambda_{k,r}((H_i)_{i=1}^l)\gtrsim k^{l-1}.
\end{equation}
\end{lemma}
\begin{proof}
Since $H_i\prec H_{i+1}$ for $1\leq i\leq l-1$, we may fix
\begin{eqnarray*}
\tau^{(i)}\in F^*\;\;{\rm with}\;\; |\tau^{(i)}|\leq N\;\;{\rm and}\;\; a_i\in H_i,\;b_i\in H_{i+1},
\end{eqnarray*}
such that $a_i\ast\tau^{(i)}\ast b_i$
is a path traveling from $H_i$ to $H_{i+1}$.

Since $H_{i+1}$ is strongly connected, we may fix words $\rho^{(i+1)}(j),j\in H_{i+1}$, such that $|\rho^{(i+1)}(j)|<N$, and $b_i\ast\rho^{(i+1)}(j)\ast j$ is a path from $b_i$ to $j$. We set
\[
\mathcal{A}_{i+1}:=\{b_i\ast\rho^{(i+1)}(j)\ast j:j\in H_{i+1}\},\;1\leq i\leq l-2.
\]

In order to show (\ref{s4}), we first construct a subset $J_{k,r}$ of $\Lambda_{k,r}(H_1,\ldots,H_l)$. We consider the set $E_q$ of words in $\Omega_q$ of the following form:
\begin{eqnarray}
&\omega=\sigma_\omega^{(1)}\ast\tau^{(1)}\ast \rho^{(2)}\ast\sigma_\omega^{(2)}\ast\tau^{(2)}\ast\rho^{(3)}\ast\sigma_\omega^{(3)}\ast\tau^{(3)}\ast\cdots\ast\sigma_\omega^{(l-1)},\label{s14}\\
&|\omega|=q,\;\;\sigma_\omega^{(i)}\in H_i^*,\;\;(\sigma_\omega^{(i)})_{|\sigma^{(i)}|}=a_i,\;\rho^{(i)}\in \mathcal{A}_i,\;\;1\leq i\leq l-1.\nonumber
\end{eqnarray}

For every $q\leq l_{1k}-N-1$ and $\omega\in E_q$, we have $p_{\omega\ast\tau^{(l-1)}} c_{\omega\ast\tau^{(l-1)}}^r\geq\underline{\eta}_r^k$; otherwise the minimal length of words in $\Lambda_{k,r}$ would be less than $l_{1k}$, which contradicts the definition of $l_{1k}$. Thus, for all $q\leq l_{1k}-N-1$ and $\omega\in E_q$, there exist a finite maximal antichain $D_q(\omega)$ in $H_l^*(b_l)$ such that $\omega\ast\tau^{(l)}\ast\sigma\in\Lambda_{k,r}$ for all $\sigma\in D_q(\omega)$. Write
\[
F_q:=\{\omega\ast\tau^{(l)}\ast\sigma\in\Lambda_{k,r}:\omega\in E_q,\sigma\in D_q(\omega)\}.
\]
Thus, for  $q\leq l_{1k}-N-1$, we have
$F_q\subset\Lambda_{k,r}(H_1,\ldots,H_l)$. Set
\[
J_{k,r}:=\bigcup_{q=3lN}^{l_{1k}-N-1}F_q.
\]
Let $I^{(2)}_{k,r}$ denote the set of all vectors $(\sigma^{(1)},\ldots,\sigma^{(l-1)})$ such that for some word $\gamma=\omega\ast\tau^{(l)}\ast\sigma\in J_{k,r}$, we have $\sigma_\omega^{(i)}=\sigma^{(i)},1\leq i\leq l-1$.

Next, we show (\ref{s4}) holds. First we note that
\[
\sum_{i=1}^{l-1}|\tau^{(i)}|+\sum_{i=1}^{l-1}\max_{\rho^{(i)}\in\mathcal{A}_i}|\rho^{(i)}|<2lN.
\]
Let $q_i,1\leq i\leq l-1$, be as defined in (\ref{s5})-(\ref{s6}). We allocate a total length $l_{1k}-3lN-1$ among $\sigma^{(i)},1\leq i\leq l$:

(i) $\sigma^{(1)}$ is allowed to starts at all $j\in H_1$ and terminates at $a_1\in H_1$; the quantity $p_1:=|\sigma^{(1)}|-N$ can take values:
$0,\ldots,q_1$;

(ii) for every $j\in H_2$, $j\ast\sigma^{(2)}$ initiates at $j$ and terminates at $a_2\in H_2$; and $p_2:=|\sigma^{(2)}|-N$ can take values£º
$0,\ldots,q_2$;

(iii) for every $j\in H_i$, $j\ast \sigma^{(i)}$ initiates at $j$ and terminates at $a_i\in H_i$; $p_i:=|\sigma^{(i)}|-N$ can have values£º
$0,\ldots,q_i,1\leq i\leq l-1$.

For $\sigma^{(i)},1\leq i\leq l-1$, satisfying (i)-(iii) and $\omega$ as defined in (\ref{s14}), we have
\begin{eqnarray*}
(l-1)N\leq |\omega|&=&\sum_{i=1}^{l-1}|\tau^{(i)}|+\sum_{i=1}^{l-1}|\rho^{(i)}|+\sum_{i=1}^{l-1}|\sigma^{(i)}|\\&\leq& 2(l-1)N+(l_{1k}-2lN-1)
\\&=& l_{1k}-2N-1.
\end{eqnarray*}
In addition, we have $\omega|_{|\omega|}=a_{l-1}$. Hence, such an $\omega$ belongs to $E_q$ with $q\leq l_{1k}-N-1$. Thus, the set of $(\sigma^{(1)},\ldots,\sigma^{(l-1)})$ for which (i)-(iii) are satisfied  is a subset of $I^{(2)}_{k,r}$. Using this fact and Lemmas \ref{g1} and \ref{zz7}, we have
\begin{eqnarray}
A_{k,r}:&=&\sum_{(\sigma^{(1)},\ldots,\sigma^{(l-1)})\in I^{(2)}_{k,r}}
\prod_{i=1}^{l-1}\sum_{j\in H_i}(p_{j\ast\sigma^{(i)}}c_{j\ast\sigma^{(i)}}^r)^{\frac{s_r}{s_r+r}}\nonumber\\&\geq&\sum_{p_1=0}^{q_1}\sum_{j\in H_1}c_{ja_1}^{(p_1+N)}(H_1)\cdots\sum_{p_{l-1}=0}^{q_{l-1}}\sum_{j\in H_{l-1}}c_{ja_{l-1}}^{(p_{l-1}+N)}(H_{l-1})
\nonumber\\&\geq& C_1^{l-1}\sum_{p_1=0}^{q_1}\sum_{p_2=0}^{q_2}\cdots\sum_{p_{l-1}=0}^{q_{l-1}}1\gtrsim  k^{l-1}.\label{s10}
\end{eqnarray}
For fixed $\omega\in E_q$ with $q\leq l_{1k}-N-1$, we have
\[
 D_q(\omega):=\{\sigma\in H^*_l(b_l):\omega\ast\tau^{(l)}\ast\sigma\in\Lambda_{k,r}\}.
 \]
Then $D_q(\omega)$ is a finite maximal antichain in $H^*_l(b_l)$.
By (\ref{s7}), we have
\begin{eqnarray}\label{s11}
\sum_{\sigma\in D_q(\omega)}(p_\sigma c_\sigma^r)^{\frac{s_r}{s_r+r}}\geq M_0.
\end{eqnarray}
Note that $|j\ast\rho^{(i+1)}(j)|\leq N$ for all $j\in H_i$ and $1\leq i\leq l$. We have
\begin{eqnarray}\label{s12}
\min_{1\leq i\leq l}(p_{\tau^{(i)}}c_{\tau^{(i)}}^r)^{\frac{s_r}{s_r+r}},\;\min_{\rho^{(i)}\in A_i}(p_{\rho^{(i)}}c_{\rho^{(i)}}^r)^{\frac{s_r}{s_r+r}}\geq\underline{\eta}_r^{\frac{Ns_r}{s_r+r}}.
\end{eqnarray}
By our construction, $\Lambda_{k,r}(H_1,\cdots,H_l)$ contains words of the following form:
\[
\sigma^{(1)}\ast\tau^{(1)}\ast \rho^{(2)}\ast\sigma^{(2)}\ast\tau^{(2)}\ast\rho^{(3)}\ast\sigma^{(3)}\ast\tau^{(3)}\ast\cdots\ast\sigma^{(l-1)}\ast\tau^{(l-1)}\ast\sigma^{(l)}.
\]
For such words $\gamma$, by (\ref{s12}), we have
\begin{eqnarray*}
(p_\gamma c_\gamma^r)^{\frac{s_r}{s_r+r}}\geq
\underline{\eta}_r^{\frac{2Nls_r}{s_r+r}}(p_{\sigma^{(l)}}c_{\sigma^{(l)}}^r)^{\frac{s_r}{s_r+r}}\prod_{i=1}^{l-1}(p_{\sigma^{(i)}}c_{\sigma^{(i)}}^r)^{\frac{s_r}{s_r+r}}.
\end{eqnarray*}
For a fixed $(\sigma^{(1)},\ldots,\sigma^{(l-1)})\in I^{(2)}_{k,r}$, we denote by $J_{k,r}(\sigma^{(1)},\ldots,\sigma^{(l-1)})$ the set of corresponding words $\gamma$ in $J_{k,r}$. Then
\begin{eqnarray*}
&&\sum_{\gamma\in J_{k,r}(\sigma^{(1)},\ldots,\sigma^{(l-1)})}(p_\gamma c_\gamma^r)^{\frac{s_r}{s_r+r}}\\&&\geq\underline{\eta}_r^{\frac{2Nls_r}{s_r+r}}\prod_{i=1}^{l-1}\sum_{j\in H_i}(p_{j\ast\sigma^{(i)}}c_{j\ast\sigma^{(i)}}^r)^{\frac{s_r}{s_r+r}}\sum_{\sigma^{(l)}\in D_q(\omega)}(p_{\sigma^{(l)}}c_{\sigma^{(l)}}^r)^{\frac{s_r}{s_r+r}}\\&&\geq M_0\underline{\eta}_r^{\frac{2Nls_r}{s_r+r}}\prod_{i=1}^{l-1}\sum_{j\in H_i}(p_{j\ast\sigma^{(i)}}c_{j\ast\sigma^{(i)}}^r)^{\frac{s_r}{s_r+r}}.
\end{eqnarray*}
Using this and (\ref{s10}), we conclude
\begin{eqnarray*}
\lambda_{k,r}((H_i)_{i=1}^l)&\geq&\sum_{\gamma\in J_{k,r}}(p_\gamma c_\gamma^r)^{\frac{s_r}{s_r+r}}\\&\geq&\sum_{(\sigma^{(1)},\ldots,\sigma^{(l-1)})\in I^{(2)}_{k,r}}\sum_{\gamma\in J_{k,r}(\sigma^{(1)},\ldots,\sigma^{(l-1)})}(p_\gamma c_\gamma^r)^{\frac{s_r}{s_r+r}}\\&\geq& M_0\underline{\eta}_r^{\frac{2Nls_r}{s_r+r}}A_{k,r}\\&\gtrsim& l_{1k}^{l-1}\asymp k^{l-1}.
\end{eqnarray*}
This completes the proof of the lemma.
\end{proof}
\begin{lemma}\label{zz5}
Let $\Lambda_{k,r}$ be as defined in (\ref{s3}). We have
\begin{equation*}
\sum_{\gamma\in\Lambda_{k,r}}(p_\gamma c_\gamma^r)^{\frac{s_r}{s_r+r}}\asymp k^{T_r-1}.
\end{equation*}
\end{lemma}
\begin{proof}
For $0\leq l\leq T_r$, we write
\[
\Lambda_{k,r}(l):=\{\gamma\in\Lambda_{k,r}:T_r(\gamma)=l\}.
\]
Then $\Lambda_{k,r}(0)$ is a subset of $F^*$. We have
\[
\Lambda_{k,r}(l)=\bigcup_{(H_1,\ldots,H_l)\in\mathcal{H}_l}\Lambda_{k,r}(H_1,\ldots,H_l).
\]
By Lemma \ref{zz6} and  Proposition 3.9 of \cite{KZ:15}, there exists a constant $M_4$ which is independent of $k$, such that
\begin{eqnarray*}
\sum_{\gamma\in\Lambda_{k,r}(0)}(p_\gamma c_\gamma^r)^{\frac{s_r}{s_r+r}}+\sum_{\gamma\in\Lambda_{k,r}(1)}(p_\gamma c_\gamma^r)^{\frac{s_r}{s_r+r}}\leq M_4.
\end{eqnarray*}
This, together with (\ref{s16}) and Lemmas \ref{zz4}, \ref{zz9}, yields
\begin{eqnarray*}
\sum_{\gamma\in\Lambda_{k,r}}(p_\gamma c_\gamma^r)^{\frac{s_r}{s_r+r}}&=&M_4+\sum_{l=2}^{T_r}\sum_{\gamma\in\Lambda_{k,r}(l)}(p_\gamma c_\gamma^r)^{\frac{s_r}{s_r+r}}\\&\lesssim&\sum_{l=2}^{T_r}\binom{T_r}{l} l_{2k}^{l-1}\lesssim l_{2k}^{T_r-1}\asymp k^{T_r-1}.
\end{eqnarray*}
On the other hand, by Lemmas \ref{zz8}, \ref{zz9}, we have
\begin{eqnarray*}
\sum_{\gamma\in\Lambda_{k,r}}(p_\gamma c_\gamma^r)^{\frac{s_r}{s_r+r}}&=&\sum_{l=0}^{T_r}\sum_{\gamma\in\Lambda_{k,r}(l)}(p_\gamma c_\gamma^r)^{\frac{s_r}{s_r+r}}\gtrsim \sum_{l=2}^{T_r}l_{1k}^{l-1}\asymp k^{T_r-1}.
\end{eqnarray*}
This completes the proof of the lemma.
\end{proof}

With the above preparations, we are now able to prove our main result.

\emph{Proof of Theorem \ref{mthm1}}

By (\ref{s3}) and Lemmas \ref{zz5}, \ref{zz9}, we have
\[
\phi_{k,r}\underline{\eta}_r^{\frac{ks_r}{s_r+r}}\asymp\sum_{\gamma\in\Lambda_{k,r}}(p_\gamma c_\gamma^r)^{\frac{s_r}{s_r+r}}\asymp k^{T_r-1}\asymp (\log\phi_{k,r})^{T_r-1}.
\]
It follows that $\underline{\eta}_r^{\frac{kr}{s_r+r}}\asymp\phi_{k,r}^{-\frac{r}{s_r}}(\log\phi_{k,r})^{\frac{r(T_r-1)}{s_r}}$. This, together with (\ref{characterization}), implies
\begin{eqnarray*}
e_{\phi_{k,r},r}^{r}(\mu)&\asymp&\sum_{\sigma\in\Lambda_{k,r}}p_{\sigma}c_{\sigma}^{r}=\sum_{\sigma\in\Lambda_{k,r}}(p_{\sigma}c_{\sigma}^{r})^{\frac{s_r}{s_r+r}}(p_{\sigma}c_{\sigma}^{r})^{\frac{r}{s_r+r}}\\
&\leq&\sum_{\sigma\in\Lambda_{k,r}}(p_{\sigma}c_{\sigma}^{r})^{\frac{s_r}{s_r+r}}\underline{\eta}_r^{\frac{kr}{s_r+r}}\\&\asymp&\phi_{k,r}^{-\frac{r}{s_r}}(\log\phi_{k,r})^{(T_r-1)(1+\frac{r}{s_r})}.
\end{eqnarray*}
By H\"{o}lder's inequality with exponent less than one, we have
\begin{eqnarray*}
e_{\phi_{k,r},r}^{r}(\mu)
&\geq&\bigg(\sum_{\sigma\in\Lambda_{k,r}}(p_{\sigma}c_{\sigma}^{r})^{\frac{s_r}{s_r+r}}\bigg)^{\frac{s_r+r}{s_r}}\phi_{k,r}^{-\frac{r}{s_r}}\asymp\phi_{k,r}^{-\frac{r}{s_r}}(\log\phi_{k,r})^{(T_r-1)(1+\frac{r}{s_r})}.
\end{eqnarray*}
For $n\geq\phi_{1,r}$, there exists a unique $k$ such that $\phi_{k,r}\leq n<\phi_{k+1,r}$. As is showed in the proof of Lemma 2.1 of \cite{KZ:15}, we have $\phi_{k,r}\asymp\phi_{k+1,r}$. Thus,
\begin{eqnarray*}
e_{n,r}^{r}(\mu)\left\{\begin{array}{ll}\leq e_{\phi_{k,r},r}^{r}(\mu)
\lesssim
n^{-\frac{r}{s_r}}(\log n)^{(T_r-1)(1+\frac{r}{s_r})}\\
\geq e_{\phi_{k+1,r},r}^{r}(\mu)
\gtrsim n^{-\frac{r}{s_r}}(\log n)^{(T_r-1)(1+\frac{r}{s_r})}\end{array}\right..
\end{eqnarray*}
This completes the proof of the theorem.

\end{document}